\documentclass{amsart}
\date{18 Jan, 2011}
\usepackage{amsmath,amsthm,amssymb,mathrsfs}
\title[Character tables of association schemes]
{Character tables of association schemes based on attenuated
spaces}
\author[H.~Kurihara]{Hirotake Kurihara}
\thanks{The author is supported by JSPS Research Fellowship.}
\address{Mathematical Institute, Tohoku University, Aoba 6-3, Sendai
980-8578, Japan}
\email{sa9d05@math.tohoku.ac.jp}

\numberwithin{equation}{section}
\newtheorem{thm}{Theorem}[section]
\newtheorem*{thm*}{Theorem}
\newtheorem{lem}[thm]{Lemma}
\newtheorem{prop}[thm]{Proposition}
\newtheorem{cor}[thm]{Corollary}
\newtheorem*{clm}{Claim}
\theoremstyle{definition}
\newtheorem{df}[thm]{Definition}
\theoremstyle{remark}

\def\set#1#2{\{#1\,|\,#2\}}
\newcommand{\trans}[1]{{}^t\hspace{-0.4ex}#1}

\newcommand{\X}{\mathcal{X}}
\newcommand{\Y}{\mathcal{Y}}

\newcommand{\fp}{\mathfrak{p}}
\newcommand{\fq}{\mathfrak{q}}

\newcommand{\qbinom}[2]{\genfrac{[}{]}{0pt}{}{#1}{#2}} 
\newcommand{\im}{\mathrm{Im}\,}
\newcommand{\rank}{\mathrm{rank}\,}
\newcommand{\tr}{\mathrm{Tr}\,}
\newcommand{\xprod}[2]{
\langle #1,#2 \rangle
}

\begin{document}
\begin{abstract}
The set of subspaces of
a given dimension in an attenuated space has a structure of
a symmetric association scheme and
this association scheme is called an association scheme based on
an attenuated space.
Association schemes based on attenuated spaces are generalizations of
Grassmann schemes and bilinear forms schemes,
and also $q$-analogues of non-binary Johnson schemes.
Wang, Guo and Li computed the intersection numbers of association
schemes based on attenuated spaces.
The aim of this paper is to compute character tables of association
schemes based on attenuated spaces using the method of
Tarnanen, Aaltonen and Goethals.
Moreover, we also prove that association schemes
based on attenuated spaces include as a special case
the $m$-flat association scheme,
which is defined on the set of
cosets of subspaces of a constant dimension in a vector space over
a finite field.
\end{abstract}

\maketitle

\section{Introduction}
\label{sec:intro}
Let $\mathbb{F}_q$ be the finite field of size $q$
and $\mathbb{F}^N_q$ denotes the vector space of
$N$-tuples over $\mathbb{F}_q$.
For a positive integer $n$ and a non-negative integer $l$, 
we fix an $l$-dimensional subspace
$\mathfrak{e}$ of $\mathbb{F}^{n+l}_q$.
The corresponding {\it attenuated space} associated with
$\mathbb{F}^{n+l}_q$ and $\mathfrak{e}$ is the collection of
all subspaces of $\mathbb{F}^{n+l}_q$ intersecting trivially
with $\mathfrak{e}$.
For non-negative integers $m$ and $k$, an $m$-dimensional subspace
$\mathfrak{p}$ of $\mathbb{F}^{n+l}_q$ is called a subspace of
{\it type $(m,k)$} with respect to $\mathfrak{e}$
if $\dim \mathfrak{p}\cap \mathfrak{e}=k$,
and especially a subspace $\mathfrak{p}$ of type $(m,0)$ is
an element of the attenuated space associated with
$\mathbb{F}^{n+l}_q$ and $\mathfrak{e}$.
Denote the set of all subspaces of
type $(m,k)$ in $\mathbb{F}^{n+l}_q$ by $\mathcal{M}_q(m,k;n+l,n)$.
The cardinality of $\mathcal{M}_q(m,k;n+l,n)$ is
\[q^{(m-k)(l-k)}\qbinom{n}{m-k}_q\qbinom{l}{k}_q,\]
where $\qbinom{n}{k}_q=\prod ^{k-1}_{i=0}\frac{q^n-q^i}{q^k-q^i}$,
i.e., the {\it Gaussian coefficient}.
The subscript $q$ will be omitted
when there is no possibility of confusion.

In 2009, Wang, Guo and Li proved that $\mathcal{M}_q(m,0;n+l,n)$ has
a structure of a symmetric association scheme
$\mathfrak{X}(\mathcal{M}_q(m,0;n+l,n))$ and
they computed intersection numbers of
$\mathfrak{X}(\mathcal{M}_q(m,0;n+l,n))$~\cite{Wang2009asb}.
In their paper, the relation $R_{(i,j)}$ on
$\mathcal{M}_q(m,0;n+l,n)$ is defined to be the set of pairs
$(\mathfrak{p},\mathfrak{q})$ satisfying
\[
\dim ((\mathfrak{p}+\mathfrak{e})/\mathfrak{e}
\cap (\mathfrak{q}+\mathfrak{e})/\mathfrak{e})
=m-i\ 
\text{and}\ 
\dim \mathfrak{p}\cap \mathfrak{q}=(m-i)-j,\]
for $(i,j)\in K=\set{(i,j)\in {\mathbb{Z}_{\ge 0}}^2}
{i\le m\wedge (n-m),
\ j\le (m-i)\wedge l}$, 
where for integers $a$ and $b$, 
the value $a\wedge b$ denotes
$\min\{a,b \}$ for short.
Then $(\mathcal{M}_q(m,0;n+l,n), \{R_{(i,j)}\}_{(i,j)\in K})$ is
a symmetric association scheme
and called an
{\it association scheme based on
an attenuated space}.
In this paper,
we denote it by $\mathfrak{X}(\mathcal{M}_q(m,0;n+l,n))$.

The association scheme $\mathfrak{X}(\mathcal{M}_q(m,0;n+l,n))$
is a common generalization of the Grassmann scheme $J_q(n,m)$
and the bilinear forms scheme $H_q(n,l)$.
In fact, if $l=0$, then the association scheme
$\mathfrak{X}(\mathcal{M}_q(m,0;n,n))$ is isomorphic to
the Grassmann scheme $J_q(n,m)$ and if $m=n$,
then the association scheme $\mathfrak{X}(\mathcal{M}_q(n,0;n+l,n))$
is isomorphic to the bilinear forms scheme $H_q(n,l)$.
Moreover the association scheme
$\mathfrak{X}(\mathcal{M}_q(m,0;n+l,n))$ is also a $q$-analogue of
the non-binary Johnson scheme (cf.~\cite{Tarnanen1985njs}).

The aim of this paper is to determine the character tables of
association schemes based on attenuated spaces.
To determine the character tables, we use the method of Tarnanen,
Aaltonen and Goethals~\cite{Tarnanen1985njs}.

Determining the character table of an association scheme corresponds
to determining the spherical functions of a compact symmetric space.
In~\cite{Bannai1984aci}, Bannai and Ito referred to an analogy
between compact symmetric spaces of rank one and  a family of
association schemes which are called ($P$ and $Q$)-polynomial
association schemes.
In fact, zonal spherical functions of compact symmetric spaces of
rank one and eigenmatrices of ($P$ and $Q$)-polynomial association
schemes are described by certain orthogonal polynomials.
Moreover Bannai~\cite{Bannai1990ctoa} had the assurance that there
exists an analogy between general compact symmetric spaces and
most of commutative association schemes, and he observed relations
between spherical functions of some compact symmetric spaces and
character tables of some commutative association schemes.
In order to study relations between compact symmetric spaces and
commutative association schemes,
it is useful to know many examples of character tables of
commutative association schemes.

In Section~\ref{sec:main_sec},
the main result of this paper will be described after giving
some definitions and basic facts about association schemes.
In Section~\ref{sec:proof_of_main}, we prove Lemma~\ref{prop:prod_c},
which is a key lemma of the proof of the main result.
In Section~\ref{sec:m-flat},
a relation between association schemes based on attenuated spaces and
$m$-flat association schemes will be found.
Namely, we prove that association schemes
based on attenuated spaces include as a special case
the $m$-flat association scheme,
which is defined on the set of
cosets of subspaces of a constant dimension in a vector space over
a finite field.
Finally in Appendix~\ref{sec:formula}
we describe some useful formulas about the number of
subspaces of a vector space over a finite field
and equations related to the $q$-Gaussian coefficient,
the generalized Eberlein polynomials and
the generalized Krawtchouk polynomials,
which are used in this paper.

\section{Character tables of association schemes based on attenuated
spaces}
\label{sec:main_sec}
We begin with a review of basic definitions concerning
association schemes.
The reader is referred to Bannai-Ito~\cite{Bannai1984aci} for the
background material.

A {\it symmetric association scheme}
$\mathfrak{X}=(X,\{R_i\}_{0\le i\le d})$
consists of a finite set $X$ and a set $\{R_i\}_{0\le i\le d}$
of binary relations on $X$ satisfying:
\begin{enumerate}
\item $R_0=\set{(x,x)}{x\in X}$;
\item $\{R_i\}_{0\le i\le d}$ is a partition of $X\times X$;
\item $\trans{R_i}=R_i$ for each $i\in \{0,1,\ldots ,d\}$,
where $\trans{R_i}=\set{(y,x)}{(x,y)\in R_i}$;
\item the numbers $|\set{z\in X}{\text{$(x,z)\in R_i$ and $(z,y)\in
R_j$}}|$ are constant whenever $(x,y)\in R_k$, for each $i,j,k\in
\{0,1,\ldots ,d\}$.
\end{enumerate}
Note that the numbers $|\set{z\in X}{\text{$(x,z)\in R_i$ and
$(z,y) \in R_j$}}|$ are called the {\it intersection numbers}
and denoted by $p^{k}_{i,j}$.
Let $M_X(\mathbb{C})$ denote the algebra of matrices over the complex
field $\mathbb{C}$ with rows and columns indexed by $X$.
The $i$-th {\it adjacency matrix} $A_i$ in $M_X(\mathbb{C})$ of
$\mathfrak{X}$ is defined by
\[A_i(x,y)=
\begin{cases}
1 & \text{if $(x,y)\in R_i$,}\\
0 & \text{otherwise.}
\end{cases}\]
The vector space
$\mathfrak{A}=\langle A_0,A_1,\ldots, A_d\rangle_\mathbb{C}$
spanned by $A_i$'s ($i=0,1,\ldots ,d$) forms a commutative algebra
and called the {\it Bose-Mesner algebra} of
$\mathfrak{X}=(X,\{R_i\}_{0\le i\le d})$.
It is well known that $\mathfrak{A}$ is semi-simple,
hence $\mathfrak{A}$ has a basis
consisting of the primitive
idempotents $E_0=1/|X|J, E_1,\ldots ,E_d$,
i.e., $E_i E_j=\delta_{i,j}E_i$, $\sum^{d}_{i=0}E_i=I$, where $J$ is
the all-one matrix and
the notation $\delta _{i,j}$ stands for the value 1 if $i=j$,
0 otherwise.
The {\it first eigenmatrix} $P=(P_i(j))_{0\le j,i\le d}$
and the {\it second eigenmatrix} $Q=(Q_i(j))_{0\le j,i\le d}$
of $\mathfrak{X}$ are defined by 
\[A_i=\sum^{d}_{j=0}P_i(j)E_j\ 
\text{and}\ 
E_i=\frac{1}{|X|}\sum^{d}_{j=0}Q_i(j)A_j,\]
respectively.
In particular, the first eigenmatrix $P$ is also called the
{\it character table} of $\mathfrak{X}$.
Note that $P$ and $Q$ satisfy $P Q=Q P=|X| I$,
and it is well known that
\begin{equation}
\label{eq:P-Q-valency-multi}
\frac{Q_j(i)}{m_j}
=
\frac{P_i(j)}{v_i},
\end{equation}
where
$v_i=p^0_{i,i}$ is called the $i$-th {\it valency}
and
$m_j=\tr (E_j)$ is called the $j$-th {\it multiplicity}.

Next we give two examples of association schemes.
The eigenmatrices of association schemes based on
attenuated spaces will be given using
the entries of
the eigenmatrices of these association schemes.
Let $V$ and $E$ be $n$-dimensional and $l$-dimensional vector spaces
over $\mathbb{F}_q$, respectively,
and let $L(V,E)$ denote the set of all linear maps from $V$ to $E$.
For a subspace $\mathfrak{p}$ of $V$, the set of all $r$-dimensional
subspaces of $\mathfrak{p}$ is denoted by $\qbinom{\mathfrak{p}}{r}$.

The set $\qbinom{V}{m}$ together with the nonempty relations
\[
R_i=\set{(\mathfrak{x},\mathfrak{y})\in {\textstyle \qbinom{V}{m}^2}}
{\dim \mathfrak{x}\cap \mathfrak{y}=m-i}
\]
is an
$m\wedge (n-m)$-class symmetric association scheme called the
{\it Grassmann scheme} $J_q(n,m)$.
The first eigenmatrix $P^G=(P^G_k(x))_{0\le x,k\le m\wedge (n-m)}$
of the Grassmann scheme
$J_q(n,m)$ is given by the {\it generalized Eberlein polynomials}
$E_k(n,m;q;x)$
\cite{Delsarte1976paa}, namely
\begin{eqnarray*}
P^G_k(x)&=&E_k(n,m;q;x)\\
&=&\sum^{k}_{j=0}(-1)^{k-j}q^{j x+\binom{k-j}{2}}\qbinom{m-j}{m-k}
\qbinom{m-x}{j}\qbinom{n-m+j-x}{j}.
\end{eqnarray*}
Furthermore, since the valencies $v_i$ and the multiplicities $m_j$
of the Grassmann scheme $J_q(n,m)$ are given as
$v_i=q^{i^2}\qbinom{n-m}{i}\qbinom{m}{i}$
and
$m_j=\qbinom{n}{j}-\qbinom{n}{j-1}$,
respectively (cf.~\cite[p.~269]{Brouwer1989dg}),
by (\ref{eq:P-Q-valency-multi})
we obtain the second eigenmatrix 
$Q^G=(Q^G_k(x))_{0\le x,k\le m\wedge (n-m)}$
of the Grassmann scheme
$J_q(n,m)$
as follows;
\begin{eqnarray*}
\lefteqn{Q^G_k(x)}\\
&=&\frac{\qbinom{n}{k}-\qbinom{n}{k-1}}
{q^{x^2}\qbinom{n-m}{x}\qbinom{m}{x}}
\sum^{x}_{j=0}(-1)^{x-j}q^{j k+\binom{x-j}{2}}\qbinom{m-j}{m-x}
\qbinom{m-k}{j}\qbinom{n-m+j-k}{j}.
\end{eqnarray*}
We denote this value by $Q_k(n,m;q;x)$.

The set $L(V,E)$ together with the nonempty relations
\[R_i=\set{(f,g)\in L(V,E)^2}{\rank (f-g)=i}\]
is an $n\wedge l$-class symmetric association scheme called the
{\it bilinear forms scheme} $H_q(n,l)$~\cite{Delsarte1978bfo}.
The first eigenmatrix $P^B=(P^B_k(x))_{0\le x,k\le n\wedge l}$
of the bilinear forms scheme
$H_q(n,l)$ is given by the {\it generalized Krawtchouk polynomials}
$K_k(n,l;q;x)$
\cite{Delsarte1976paa}, namely
\begin{eqnarray*}
P^B_k(x)&=&K_k(n,l;q;x)\\
&=&\sum^{k}_{j=0}(-1)^{k-j}q^{j l+\binom{k-j}{2}}\qbinom{n-j}{n-k}
\qbinom{n-x}{j}.
\end{eqnarray*}
Note that bilinear forms schemes are self-dual, i.e., $P^B=Q^B$,
where $Q^B$ is the second eigenmatrix of $H_q(n,l)$
(see~\cite{Delsarte1978bfo}).
Hence, by $P^B Q^B =|L(V,E)| I$, we obtain
\begin{equation}
\label{eq:orth_Krawtchouk}
\sum^{n \wedge l}_{k=0}
K_k(n,l;q;i)
K_j(n,l;q;k)
=q^{n l} \delta _{i,j}.
\end{equation}

In order to calculate the character table of the association scheme
$\mathfrak{X}(\mathcal{M}_q(m,0;n+l,n))$,
we deal with another realization of
$\mathfrak{X}(\mathcal{M}_q(m,0;n+l,n))$.
Let $\X$ be the set of all pairs of a subspace $\mathfrak{x}$ in $V$
and a linear map $\xi$ in $L(\mathfrak{x}, E)$ and let
$\X_m =\set{(\mathfrak{x},\xi)\in \X}{\mathfrak{x}\in \qbinom{V}{m}}$.
The set $\X_m$ has cardinality $q^{ml} \qbinom{n}{m}$.
Then there is the following one-to-one correspondence between
$\X_m$ and $\mathcal{M}_q(m,0;n+l,n)$:
we regard $V\oplus E$ as $\mathbb{F}_q ^{n+l}$
and also regard $\{0\} \oplus E$ as $\mathfrak{e}$.
For $(\mathfrak{x},\xi)$ in $\X _m$,
we set $\mathfrak{x}_\xi =\set{(x, \xi (x))}{x\in \mathfrak{x}}$
in $\mathbb{F}_q ^{n+l}$.
Then we immediately check that
$\mathfrak{x}_\xi$ is type $(m,0)$ with respect to
$\mathfrak{e}$,
i.e., $\mathfrak{x}_\xi \in \mathcal{M}_q(m,0;n+l,n)$,
and
the map $(\mathfrak{x},\xi) \mapsto \mathfrak{x}_\xi$
from $\X _m$ to $\mathcal{M}_q(m,0;n+l,n)$
is injective.
On the other hand, we have
$|\X _m|=|\mathcal{M}_q(m,0;n+l,n)| =q^{ml} \qbinom{n}{m}$.
This means that the map from $\X _m$ to $\mathcal{M}_q(m,0;n+l,n)$
is bijective.
Next, we define the relation $S_{(i,j)}$ on $\X_m$ 
to be the set of pairs
$((\mathfrak{x},\xi), (\mathfrak{y},\eta))$ satisfying
\[
\dim \mathfrak{x}\cap \mathfrak{y}=m-i\ \text{and}\ 
\rank (\xi |_{\mathfrak{x}\cap \mathfrak{y}}-\eta |_{\mathfrak{x}
\cap \mathfrak{y}})=j.
\]
Then, by the two equalities
\begin{eqnarray*}
\dim
(\mathfrak{x}_\xi +\mathfrak{e})/\mathfrak{e}
\cap
(\mathfrak{y}_\eta +\mathfrak{e})/\mathfrak{e}
&=&
\dim
(\mathfrak{x} \cap \mathfrak{y}\oplus E)/
(\{0\}\oplus E)\\
&=&\dim \mathfrak{x} \cap \mathfrak{y}
\end{eqnarray*}
and
\begin{eqnarray*}
\dim \mathfrak{x}_\xi \cap \mathfrak{y}_\eta
&=&
\dim \set{(x, \xi(x))}{x \in \ker 
(\xi |_{\mathfrak{x}\cap \mathfrak{y}}-\eta |_{\mathfrak{x}
\cap \mathfrak{y}})}\\
&=&\dim \ker 
(\xi |_{\mathfrak{x}\cap \mathfrak{y}}-\eta |_{\mathfrak{x}
\cap \mathfrak{y}})\\
&=&\dim \mathfrak{x} \cap \mathfrak{y}-
\rank
(\xi |_{\mathfrak{x}\cap \mathfrak{y}}-\eta |_{\mathfrak{x}
\cap \mathfrak{y}}),
\end{eqnarray*}
it follows that
$((\mathfrak{x},\xi), (\mathfrak{y},\eta)) \in S_{(i,j)}$
if and only if $(\mathfrak{x}_\xi, \mathfrak{y}_\eta) \in R_{(i,j)}$.
Consequently, we obtain another representation of
$\mathfrak{X}(\mathcal{M}_q(m,0;n+l,n))$,
and we denote
the association scheme
$(\X_m, \{S_{(i,j)}\}_{(i,j)\in K})$ by
$\mathfrak{X}(\mathcal{M}_q(m,0;n+l,n))$. 

We calculate the character table of this association scheme
using the following theorem
proved by Tarnanen, Aaltonen and Goethals~\cite{Tarnanen1985njs}:
\begin{thm}
\label{thm:chara_thm}
Let $X$ be a non-empty finite set.
Assume that $\{R_i\}_{0\le i\le d}$ is a partition of $X\times X$,
each $R_i$ is a symmetric relation,
that is $\trans{R_i}=R_i$,
and $R_0=\set{(x,x)}{x\in X}$.
Let $\mathfrak{A}=\langle A_0,A_1,\ldots, A_d\rangle_\mathbb{C}$
be the complex linear space generated by the adjacency matrices
$A_i$ of $R_i$ $(i\in \{0,1,\ldots ,d\})$.
If there exist matrices $C_0=J, C_1,\ldots ,C_d$ of $\mathfrak{A}$
such that
\[A_i=\sum^{d}_{j=0}\alpha_i(j)C_j\ \text{and }\  C_k C_s=
\sum^{k\wedge s}_{j=0} \beta_{k;s}(j)C_j,\quad k,s\in
\{0,1,\ldots ,d\},\]
where $\alpha_i(j)$and $\beta_{k;s}(j)$ are complex numbers,
then $(X,\{R_i\}_{0\le i\le d})$ is a symmetric association scheme
and its first eigenmatrix $P=(P_i(s))_{0\le s, i\le d}$
is given by
\[P_i(s)=\sum^{d}_{k=s}\alpha_i(k)\beta_{k;s}(s),
\quad i,s\in \{0,1,\ldots ,d\}.\]
\end{thm}
For the rest of this paper, let $d=m\wedge (n-m)$.
Let $A_{(i,j)}$ denote 
the adjacency matrix of the relation
$S_{(i,j)}$.
We define $L=\set{(r,s)\in {\mathbb{Z}_{\ge 0}}^2}{s\le m\wedge l,
r\le m, \ 0\le r-s\le d}$.
Since $(i,j)\mapsto (r,s)=(i+j,j)$ is a bijection from
$K$ to $L$, we have $|L|=|K|$. 
For $(r,s)\in L$, a matrix $C_{(r,s)}$ is defined as follows,
which serves as $C_i$ in Theorem \ref{thm:chara_thm}:
for $0\le i\le d$ and $0\le s\le m\wedge l$,
\[B_{(i,s)}=\sum^{(m-i) \wedge l}_{j=0}K_s(m-i,l;q;j)A_{(i,j)}\]
and for $(r,s)\in L$,
\[C_{(r,s)}=\sum^{d\wedge (m-s)}_{i=0}\qbinom{m-i-s}{r-s}B_{(i,s)}.\]
Namely,
for $((\mathfrak{x},\xi),(\mathfrak{y},\eta))\in S_{(m-u,u-v)}$,
the $((\mathfrak{x},\xi), (\mathfrak{y},\eta))$-entry of $C_{(r,s)}$
is as follows:
\begin{equation}
\label{eq:C_rs}
C_{(r,s)}((\mathfrak{x},\xi),(\mathfrak{y},\eta))=
\qbinom{u-s}{r-s}K_s(u,l;q;u-v).
\end{equation}
By using (\ref{eq:orth_Krawtchouk})
for the bilinear forms scheme $H_q(m-i,l)$,
for $(i,j)\in K$,
we obtain
\begin{align*}
A_{(i,j)}&=
q^{-(m-i)l}\sum^{(m-i)\wedge l}_{h=0}K_{j}(m-i,l;q;h)B_{(i,h)}
\\ &=
q^{-(m-i)l}\sum^{m \wedge l}_{h=0}K_{j}(m-i,l;q;h)B_{(i,h)},
\end{align*}
since $K_{j}(m-i,l;q;h)=0$ if $h>m-i$.
We claim that the square matrix
\[
M^s=
\left(\qbinom{m-i-s}{r-s}\right)
_{\substack{0\le i\le d\wedge (m-s),\\ s\le r\le (d+s)\wedge m}}
\]
is nonsingular. Indeed,
since $\qbinom{m-i-s}{r-s}$ is a polynomial of degree $r-s$
in $q^{-i}$,
$M^s$ can be converted into a Vandermonde matrix
$((q^{-i})^{r'})_{0\le i,r'\le d\wedge (m-s)}$
by a sequence of elementary transformations.
Thus the inverse matrix $N^s$
of $M^s$ satisfies
\begin{equation}
\label{eq:N^s}
\sum^{(d+s)\wedge m}_{k=s}
\qbinom{m-i-s}{k-s}N^s (k,j)
=\delta_{i,j}
\end{equation}
for $0\le i,j\le d\wedge (m-s)$.
By (\ref{eq:N^s}),
for $0\le i\le d$ and $0\le h \le m \wedge l$,
we obtain
\[B_{(i,h)}=\sum^{(d+h)\wedge m}_{k=h}
N^h (k,i) C_{(k,h)}.\]
Therefore we have
$A_{(i,j)}=\sum_{(k,h)\in L}\alpha _{(i,j)}(k,h)C_{(k,h)}$, where
\begin{equation}
\label{eq:alpha}
\alpha _{(i,j)}(k,h)=q^{-(m-i)l}
K_{j}(m-i,l;q;h) N^h (k,i).
\end{equation}

The following lemma gives the expansion of $C_{(r,s)}C_{(k,h)}$
in terms of $C_{(i,j)}$,
which serves as the expansion of $C_k C_s$ in terms of $C_j$
in Theorem \ref{thm:chara_thm}.
\begin{lem}
\label{prop:prod_c}
The matrices $\{C_{(r,s)}\}_{(r,s)\in L}$ satisfy
\[C_{(r,s)}C_{(k,h)}=\sum^{r\wedge k}_{i=s}
\beta_{(r,s;k,h)}(i,s) C_{(i,s)},\]
where
\[\beta_{(r,s;k,h)}(i,s)=\delta_{s,h}q^{ml+(i-s)(m-r-k+i)}
\qbinom{m-i}{m-r}\qbinom{m-i}{m-k}\qbinom{n-r-k+s}{m-r-k+i}.\]
\end{lem}
We will prove Lemma \ref{prop:prod_c} in Section
\ref{sec:proof_of_main}.
Using Lemma \ref{prop:prod_c},
we obtain the main theorem of this paper.
\begin{thm}
\label{thm:main}
Let $P=(P_{(i,j)}(r,s))_{(r,s)\in L, (i,j)\in K}$
and $Q=(Q_{(r,s)}(i,j))_{(i,j)\in K, (r,s)\in L}$
be the first and second eigenmatrices of the association scheme based
on attenuated space $\mathfrak{X}(\mathcal{M}_q(m,0;n+l,n))$,
respectively.
Then the following hold.
\begin{equation}
\label{eq:main_thm_P}
P_{(i,j)}(r,s)=
q^{i l}K_{j}(m-i,l;q;s)E_i(n-s,m-s;q;r-s)
\end{equation}
and
\begin{equation}
\label{eq:main_thm_Q}
Q_{(r,s)}(i,j)=
\frac{\qbinom{n}{m}}{\qbinom{n-s}{m-s}}K_s(m-i,l;q;j)
Q_{r-s}(n-s,m-s;q;i).
\end{equation}
\end{thm}
\begin{proof}
Let $\succeq$ be the lexicographical order of $L$ so that the
relation $(k,h)\succeq (r,s)$ means that $h> s$ or that $k\ge r$
if $h=s$.
By Theorem \ref{thm:chara_thm}, the eigenvalue $P_{(i,j)}(r,s)$ of
$\mathfrak{X}(\mathcal{M}_q(m,0;n+l,n))
=(\X_m, \{S_{(i,j)}\}_{(i,j)\in K})$ is calculated
as follows.
For $(i,j)\in K$ and $(r,s)\in L$,
we have
\begin{eqnarray*}
P_{(i,j)}(r,s) &=&\sum_{(k,h)\succeq (r,s)}\alpha_{(i,j)}(k,h)
\beta_{(k,h;r,s)}(r,s)\\
&=&\sum^{m\wedge (d+s)}_{k=r}q^{-(m-i)l}
K_{j}(m-i,l;q;s) N^s (k,i)\\
&&\times q^{ml}q^{(r-s)(m-k)}\qbinom{m-r}{m-k}\qbinom{n-k-r+s}{m-k}\\
&=&q^{il}K_{j}(m-i,l;q;s)\sum^{m\wedge (d+s)}_{k=s} N^s (k,i)\\
&&\times q^{(m-k)(r-s)}\qbinom{m-r}{m-k}\qbinom{n-k-r+s}{m-k},
\end{eqnarray*}
since $\qbinom{m-r}{m-k}=0$ if $k<r$.
Moreover using Lemma \ref{lem:eberlein}
and (\ref{eq:N^s}),
we obtain
\begin{eqnarray*}
P_{(i,j)}(r,s)
&=&q^{i l}K_{j}(m-i,l;q;s)\sum^{m\wedge (d+s)}_{k=s}N^s (k,i)\\
&&\times \sum^{m-s}_{t=0}\qbinom{m-s-t}{k-s}E_t(n-s,m-s;q;r-s)\\
&=&q^{i l}K_{j}(m-i,l;q;s)E_i(n-s,m-s;q;r-s).
\end{eqnarray*}

To verify
(\ref{eq:main_thm_Q}),
it is sufficient to show that
\[\sum_{(i,j)\in K}
\frac{\qbinom{n}{m}}{\qbinom{n-s}{m-s}}K_s(m-i,l;q;j)
Q_{r-s}(n-s,m-s;q;i)P_{(i,j)}(k,h)
=q^{ml}\qbinom{n}{m}\delta_{r,k}\delta_{s,h}\]
for all $(r,s)$ and $(k,h)$ in $L$.
Substituting (\ref{eq:main_thm_P}) for $P_{(i,j)}(k,h)$ and
using the orthogonality of
eigenmatrices of Grassmann schemes
$J_q(n-s,m-s)$
and bilinear forms schemes $H_q(m-i,l)$,
this orthogonality relation is verified.
\end{proof}

\section{Proof of Lemma \ref{prop:prod_c}}
\label{sec:proof_of_main}
Let $\Y$ be the set of all pairs of a subspace $\mathfrak{a}$ in $V$
and a linear map $f$ from $E$ to $\mathfrak{a}$, and
$\Y_r=\set{(\mathfrak{a},f)\in \Y}{\dim \mathfrak{a}=r}$.
For a subspace $\mathfrak{x}$ of $V$,
let $\Y^\mathfrak{x}_r=\set{(\mathfrak{a},f)\in \Y_r}{\mathfrak{a}
\subset \mathfrak{x}}$ and 
$\Y^\mathfrak{x}_{r,s}=\set{(\mathfrak{a},f)\in 
\Y^\mathfrak{x}_r}{\rank f=s}$.
We shall define a complex-valued function
$\xprod{\cdot}{\cdot}$
over $\Y\times \X$ needed in the proof of Lemma \ref{prop:prod_c}.
Consider a pair $((\mathfrak{a},f),(\mathfrak{x},\xi))\in \Y\times\X$
satisfying $\mathfrak{a}\subset \mathfrak{x}$.
When $\dim \mathfrak{x}=r$, the linear map
$f\circ \xi:\mathfrak{x}\rightarrow \mathfrak{a}\subset \mathfrak{x}$
is regarded as a square matrix $T_{f\circ \xi}$ of degree $r$ for
some basis of $\mathfrak{x}$.
Note that the trace $\tr (T_{f\circ \xi})$ of $T_{f\circ \xi}$
is independent of the chosen basis of $\mathfrak{x}$,
so that $\tr (T_{f\circ \xi})$ only depends on the pair
$((\mathfrak{a},f),(\mathfrak{x},\xi))$.
\begin{df}
Let $p$ be the prime divisor of $q$ and let
$\epsilon$ be a primitive $p$-th root of unity.
We take $\chi : \mathbb{F}_q \rightarrow \mathbb{Z}[\epsilon]$ to be a
non-principal character of the elementary abelian $p$-group
$(\mathbb{F}_q, +)$.
The function 
$\Y\times \X\ni ((\mathfrak{a},f),(\mathfrak{x},\xi))\mapsto
\xprod{(\mathfrak{a},f)}{(\mathfrak{x},\xi)}\in \mathbb{C}$
is defined as
\[\xprod{(\mathfrak{a},f)}{(\mathfrak{x},\xi)}=
\begin{cases}
\chi (\tr (T_{f\circ \xi})) &
\text{if $\mathfrak{a}\subset \mathfrak{x}$,}\\
0& \text{otherwise.}
\end{cases}\]
\end{df}
Fix an element $\mathfrak{x}$ in $\qbinom{V}{r}$.
For $f\in L(E, \mathfrak{x})$,
we define the function
$\Lambda _f : L(\mathfrak{x},E)\rightarrow \mathbb{C}$
by
$\Lambda _f(\xi)=\xprod{(\mathfrak{x},f)}{(\mathfrak{x},\xi)}$.
Then
$\set{\Lambda _f}{f\in L(E, \mathfrak{x})}$
is the character group of $L(\mathfrak{x},E)$
and we denote it by $L(\mathfrak{x},E)^\ast$.
There is the orthogonality relation of
$L(\mathfrak{x},E)^\ast$:
\begin{equation}
\label{eq:krawtchouk}
\sum_{\xi \in L(\mathfrak{x},E)}
\overline{\xprod{(\mathfrak{x},f)}{(\mathfrak{x},\xi)}}
\xprod{(\mathfrak{x},g)}{(\mathfrak{x},\xi)}=
\begin{cases}
q^{r l} & \text{if $f =g$,}\\
0 & \text{otherwise,}
\end{cases}
\end{equation}
and there is the following equation
(cf.~\cite{Delsarte1978bfo}):
\begin{equation}
\label{eq:krawtchouk_rank_s}
\sum_{\substack{f\in L(E,\mathfrak{x})\\
\rank f=s}}
\xprod{(\mathfrak{x},f)}{(\mathfrak{x},\xi)}
\overline{\xprod{(\mathfrak{x},f)}{(\mathfrak{x},\eta)}}
=K_s(r,l;q;\rank(\xi-\eta)).
\end{equation}

For $(\mathfrak{a},f)$ and $(\mathfrak{b},g)$ in $\Y$,
if $f(x)=g(x)$ for any $x\in E$,
then we call the pair $(\mathfrak{a},f)$ and $(\mathfrak{b},g)$
{\it almost equal} and denote this by
$(\mathfrak{a},f)\approx (\mathfrak{b},g)$.
When $(\mathfrak{a},f)\approx (\mathfrak{b},g)$,
it can be easily seen that
$\im f=\im g\subset \mathfrak{a}\cap \mathfrak{b}$ and
$\xprod{(\mathfrak{a},f)}{(\mathfrak{x},\xi)} =
\xprod{(\mathfrak{b},g)}{(\mathfrak{x},\xi)}$
for any $(\mathfrak{x},\xi)\in \X$ with
$\mathfrak{a}+\mathfrak{b}\subset \mathfrak{x}$.

\begin{lem}
\label{clm:chara_and_krawtchouk}
For $(\mathfrak{x},\xi)$ and $(\mathfrak{y},\eta)$ in $\X_m$,
\begin{equation}
\label{eq:vipeq}
\sum_{(\mathfrak{a},f)\in \Y^V_{r,s}}\xprod{(\mathfrak{a},f)}
{(\mathfrak{x},\xi)}\overline{\xprod{(\mathfrak{a},f)}
{(\mathfrak{y},\eta)}}
=C_{(r,s)}((\mathfrak{x},\xi), (\mathfrak{y},\eta)).
\end{equation}
\end{lem}

\begin{proof}
By the definition of $\xprod{\cdot}{\cdot}$,
the range of the summation of $(\mathfrak{a},f)$
in the left hand side of (\ref{eq:vipeq}) is restricted within
$\Y^{\mathfrak{x}\cap \mathfrak{y}}_{r,s}$.
Since $f\circ \xi$ has a matrix representation
of the following form
\[T_{f\circ \xi} =\left(
\begin{array}{cc}
T &\ast \\
0 &0 \\
\end{array}\right),\]
where $T$ is a matrix representation of
$f\circ \xi |_\mathfrak{a} \in L(\mathfrak{a},\mathfrak{a})$,
we obtain
\[\xprod{(\mathfrak{a},f)}{(\mathfrak{x},\xi)}=
\xprod{(\mathfrak{a},f)}{(\mathfrak{a},\xi |_{\mathfrak{a}})}.\]
Similarly, we obtain
\[\xprod{(\mathfrak{a},f)}{(\mathfrak{y},\eta)}=
\xprod{(\mathfrak{a},f)}{(\mathfrak{a},\eta |_{\mathfrak{a}})}.\]
Therefore by (\ref{eq:krawtchouk_rank_s}), we obtain
\begin{eqnarray*}
\sum_{(\mathfrak{a},f)\in \Y^V_{r,s}}\xprod{(\mathfrak{a},f)}
{(\mathfrak{x},\xi)}\overline{\xprod{(\mathfrak{a},f)}
{(\mathfrak{y},\eta)}}&=&
\sum_{(\mathfrak{a},f)\in \Y^{\mathfrak{x}\cap \mathfrak{y}}_{r,s}}
\xprod{(\mathfrak{a},f)}{(\mathfrak{a},\xi |_{\mathfrak{a}})}
\overline{\xprod{(\mathfrak{a},f)}
{(\mathfrak{a},\eta |_{\mathfrak{a}})}}\\
&=&\sum_{\mathfrak{a}\in \qbinom{\mathfrak{x}\cap \mathfrak{y}}{r}}
K_s(r,l;q;\rank (\xi |_{\mathfrak{a}}-\eta |_{\mathfrak{a}})).\\
\end{eqnarray*}
Let us remark that
$\rank(\xi |_{\mathfrak{a}}-\eta |_{\mathfrak{a}})=\dim \mathfrak{a}
-\dim \ker (\xi |_{\mathfrak{a}}-\eta |_{\mathfrak{a}})$
and
$\ker (\xi |_{\mathfrak{a}}-\eta |_{\mathfrak{a}})=\ker (\xi |_
{\mathfrak{x}\cap \mathfrak{y}}-\eta |_{\mathfrak{x}\cap
\mathfrak{y}})\cap \mathfrak{a}$.
Thus when $\dim \mathfrak{x}\cap
\mathfrak{y}=u$ and
$\dim \ker (\xi |_
{\mathfrak{x}\cap \mathfrak{y}}-\eta |_{\mathfrak{x}\cap
\mathfrak{y}})=v$
(i.e., $((\mathfrak{x},\xi), (\mathfrak{y},\eta)) \in S_{(m-u,u-v)}$),
by Corollary \ref{lem:BCN},
the number of
$\mathfrak{a}\in \qbinom{\mathfrak{x}\cap \mathfrak{y}}{r}$
with
$\dim \ker (\xi |_{\mathfrak{x}\cap \mathfrak{y}}-\eta |
_{\mathfrak{x}\cap \mathfrak{y}})\cap \mathfrak{a}=r-j$
is equal to 
\[q^{j(v-r+j)}\qbinom{u-v}{j}\qbinom{v}{r-j}.\]
Hence it follows that
\begin{eqnarray*}
\sum_{\mathfrak{a}\in \qbinom{\mathfrak{x}\cap \mathfrak{y}}{r}}
K_s(r,l;q;\rank (\xi |_{\mathfrak{a}}-\eta |_{\mathfrak{a}}))
&=&\sum^{r\wedge (u-v)}_{j=0\vee (v-r)}q^{j(v-r+j)}\qbinom{u-v}{j}
\qbinom{v}{r-j}K_s(r,l;q;j)\\
&=&\qbinom{u-s}{r-s}K_s(u,l;q;u-v)\\
&=&C_{(r,s)}((\mathfrak{x},\xi), (\mathfrak{y},\eta)),
\end{eqnarray*}
where for integers $a$ and $b$, 
the value $a\vee b$ denotes
$\max\{a,b \}$ for short.
In the second line,
we used Lemma \ref{lem:relation_of_krawtchouk}.
\end{proof}
\begin{proof}[Proof of Lemma \ref{prop:prod_c}]
Since $\mathfrak{X}(\mathcal{M}_q(m,0;n+l,n))$ is symmetric,
without loss of generality,
we may suppose $r\le k$.
By Lemma \ref{clm:chara_and_krawtchouk},
for $((\mathfrak{x},\xi), (\mathfrak{y},\eta))\in S_{(m-u,u-v)}$,
we obtain
\begin{eqnarray*}
C_{(r,s)}C_{(k,h)}((\mathfrak{x},\xi), (\mathfrak{y},\eta)) &=&
\sum_{(\mathfrak{z},\zeta)\in \X_m}C_{(r,s)}((\mathfrak{x},\xi),
(\mathfrak{z},\zeta))\cdot C_{(k,h)}((\mathfrak{z},\zeta),
(\mathfrak{y},\eta))\\
&=&\sum_{(\mathfrak{a},f)\in \Y^V_{r,s}}\sum_{(\mathfrak{b},g)\in
\Y^V_{k,h}}\xprod{(\mathfrak{a},f)}{(\mathfrak{x},\xi)}
\overline{\xprod{(\mathfrak{b},g)}{(\mathfrak{y},\eta)}}\\
&&\times \sum_{(\mathfrak{z},\zeta)\in \X_m}\overline{\xprod{(
\mathfrak{a},f)}{(\mathfrak{z},\zeta)}}\xprod{(\mathfrak{b},g)}
{(\mathfrak{z},\zeta)}.
\end{eqnarray*}
By the definition of $\xprod{\cdot}{\cdot}$, we may restrict the range of
the summation to be those
$(\mathfrak{z},\zeta)\in \X_m$ satisfying
$\mathfrak{z}\supset \mathfrak{a+b}$.
Thus we obtain
\begin{eqnarray*}
\sum_{(\mathfrak{z},\zeta)\in \X_m}\overline{\xprod{(\mathfrak{a},f)}
{(\mathfrak{z},\zeta)}}\xprod{(\mathfrak{b},g)}{(\mathfrak{z},\zeta)}
&=&\sum_{\substack{\mathfrak{z}\in \qbinom{V}{m}\\
\mathfrak{z}\supset \mathfrak{a}+\mathfrak{b}}}
\sum_{\zeta\in L(\mathfrak{z},E)}\overline{\xprod{(\mathfrak{a},f)}
{(\mathfrak{z},\zeta)}}\xprod{(\mathfrak{b},g)}{(\mathfrak{z},\zeta)}
\\
&=&\sum_{\substack{\mathfrak{z}\in \qbinom{V}{m}\\
\mathfrak{z}\supset \mathfrak{a}+\mathfrak{b}}}
\sum_{\zeta\in L(\mathfrak{z},E)}\overline{\xprod{(\mathfrak{z},f)}
{(\mathfrak{z},\zeta)}}\xprod{(\mathfrak{z},g)}{(\mathfrak{z},\zeta)}
\end{eqnarray*}
and, by (\ref{eq:krawtchouk}),
the value
$\sum_{\zeta\in L(\mathfrak{z},E)}\overline{\xprod{(\mathfrak{z},f)}
{(\mathfrak{z},\zeta)}}\xprod{(\mathfrak{z},g)}
{(\mathfrak{z},\zeta)}$
is equal to $q^{ml}$ if $(\mathfrak{a},f)\approx (\mathfrak{b},g)$
or vanishes otherwise.
Hence, if $h \neq s$, then $C_{(r,s)}C_{(k,h)}=0$
and we will henceforth assume that $h=s$.
The number of $\mathfrak{z}\in \qbinom{V}{m}$
satisfying
$\mathfrak{z}\supset \mathfrak{a}+\mathfrak{b}$ is equal to
$
\qbinom{n-\dim (\mathfrak{a}+\mathfrak{b})}
{m-\dim (\mathfrak{a}+\mathfrak{b})}
$.
Thus we obtain
\begin{eqnarray}
\label{eq:C_rs*C_ks}
\lefteqn{C_{(r,s)}C_{(k,s)}
((\mathfrak{x},\xi), (\mathfrak{y},\eta))}\\
\notag&=&
\sum_{\substack{(\mathfrak{a},f)\in \Y^\mathfrak{x}_{r,s}\\
(\mathfrak{b},g)\in \Y^\mathfrak{y}_{k,s}\\
(\mathfrak{a},f)\approx (\mathfrak{b},g)}}
\xprod{(\mathfrak{a},f)}{(\mathfrak{x},\xi)}
\overline{\xprod{(\mathfrak{b},g)}{(\mathfrak{y},\eta)}}
q^{ml}\qbinom{n-\dim (\mathfrak{a}+\mathfrak{b})}
{m-\dim (\mathfrak{a}+\mathfrak{b})}.
\end{eqnarray}

Let us impose the additional constraints:
$\dim\mathfrak{a}\cap \mathfrak{b}=e$,
$\dim\mathfrak{a}\cap \mathfrak{y}=r-t$ and 
$\dim\mathfrak{b}\cap \mathfrak{x}=k-j$.
Since $(\mathfrak{a},f)\approx (\mathfrak{b},g)$,
i.e., $\im f=\im g\subset \mathfrak{a} \cap \mathfrak{b}$,
it follows that $e \ge s$ and
that $C_{(r,s)}C_{(k,h)}((\mathfrak{x},\xi), (\mathfrak{y},\eta))$
is equal to 
\[
q^{ml}\sum^{r}_{e=s}\qbinom{n-r-k+e}{m-r-k+e}
\sum^{r-e}_{t=0}\sum^{k-e}_{j=0}\sum_{(\mathfrak{a},f)\in T_1}
\sum_{\mathfrak{b}\in T_2}
\xprod{(\mathfrak{a},f)}{(\mathfrak{x},\xi)}
\overline{
\xprod{(\mathfrak{b},f)}{(\mathfrak{y},\eta)}},
\]
where
$T_1=\set{(\mathfrak{a},f)\in \Y^{\mathfrak{x}}_{r,s}}
{\dim\mathfrak{a}\cap \mathfrak{y}=r-t,
\ \im f\subset \mathfrak{a}\cap \mathfrak{y}}$
and
$T_2=
\set{\mathfrak{b}\in \qbinom{\mathfrak{y}}{k}}{
\dim\mathfrak{b}\cap \mathfrak{x}=k-j,
\ \dim\mathfrak{a}\cap \mathfrak{b}=e,\ \mathfrak{b}\supset \im f}$.
Since
$\xprod{(\mathfrak{b},f)}{(\mathfrak{y},\eta)}
=
\xprod{(\mathfrak{a}\cap \mathfrak{y},f)}{(\mathfrak{y},\eta)}$
is independent of $\mathfrak{b}$,
it follows that
\begin{eqnarray}
\label{eq:(a,f)_and_b}
\lefteqn{
\sum_{(\mathfrak{a},f)\in T_1}
\sum_{\mathfrak{b}\in T_2}
\xprod{(\mathfrak{a},f)}{(\mathfrak{x},\xi)}
\overline{
\xprod{(\mathfrak{b},f)}{(\mathfrak{y},\eta)}}
}\\
&=&\notag
\sum_{(\mathfrak{a},f)\in T_1}|T_2|
\xprod{(\mathfrak{a},f)}{(\mathfrak{x},\xi)}
\overline{
\xprod{(\mathfrak{a} \cap \mathfrak{y},f)}{(\mathfrak{y},\eta)}}.
\end{eqnarray}
We shall count the number of elements in $T_2$.
By Proposition \ref{prop:dunkl},
the number of
$\mathfrak{b'}\in \qbinom{\mathfrak{x}\cap \mathfrak{y}}{k-j}$
satisfying $\mathfrak{b}'\supset \im f$ and
$\dim\mathfrak{a}\cap \mathfrak{b}'=e$ is
\[q^{(r-t-e)(k-j-e)}\qbinom{r-t-s}{e-s}\qbinom{u-r+t}{k-j-e}\]
and the number of $\mathfrak{b}\in \qbinom{\mathfrak{y}}{k}$
satisfying $\mathfrak{b}\cap \mathfrak{x}=\mathfrak{b}'$ is
\[q^{j(u-k+j)}\qbinom{m-u}{j},\]
that is,
\begin{equation}
\label{eq:T_2}
|T_2|=q^{(r-t-e)(k-j-e)+j(u-k+j)}
\qbinom{r-t-s}{e-s}\qbinom{u-r+t}{k-j-e}\qbinom{m-u}{j}.
\end{equation}
This means that
$|T_2|$ is independent of the choice of $(\mathfrak{a},f)$.
Note that, for 
$(\mathfrak{a}',f')\in \Y^{\mathfrak{x}\cap \mathfrak{y}}_{r-t,s}$,
all $(\mathfrak{a},f)\in T_1$ with
$\mathfrak{a} \cap \mathfrak{y}=\mathfrak{a}'$ and
$f(x)= f'(x)$ for any
$x\in E$ satisfy $\xprod{(\mathfrak{a},f)}{(\mathfrak{x},\xi)}
\overline{
\xprod{(\mathfrak{a} \cap \mathfrak{y},f)}{(\mathfrak{y},\eta)}}=
\xprod{(\mathfrak{a}',f')}{(\mathfrak{x},\xi)}
\overline{\xprod{(\mathfrak{a}',f')}{(\mathfrak{y},\eta)}}$.
On the other hand, by Proposition \ref{prop:dunkl},
it follows that the number of $(\mathfrak{a},f)\in T_1$ with
$\mathfrak{a} \cap \mathfrak{y}=\mathfrak{a}'$ and
$f(x)= f'(x)$ for any $x\in E$ is 
\[q^{t(u-r+t)}\qbinom{m-u}{t}.\]
Therefore we obtain
\begin{eqnarray}
\label{eq:a_and_f}
\lefteqn{\sum_{(\mathfrak{a},f)\in T_1}
\xprod{(\mathfrak{a},f)}{(\mathfrak{x},\xi)}
\overline{
\xprod{(\mathfrak{a} \cap \mathfrak{y},f)}{(\mathfrak{y},\eta)}}}\\
&=&q^{t(u-r+t)}\qbinom{m-u}{t}
\sum_{(\mathfrak{a}',f')\in \Y^{\mathfrak{x}\cap
\mathfrak{y}}_{r-t,s}}
\xprod{(\mathfrak{a}',f')}{(\mathfrak{x},\xi)}
\overline{\xprod{(\mathfrak{a}',f')}{(\mathfrak{y},\eta)}}\notag.
\end{eqnarray}
By (\ref{eq:(a,f)_and_b}), (\ref{eq:T_2}) and
(\ref{eq:a_and_f}) and Lemma \ref{clm:chara_and_krawtchouk},
it follows that
$C_{(r,s)}C_{(k,h)}((\mathfrak{x},\xi), (\mathfrak{y},\eta))$
is equal to 
\begin{eqnarray*}
\lefteqn{q^{ml}\sum^{r}_{e=s}
\qbinom{n-r-k+e}{m-r-k+e}\sum^{r-e}_{t=0}
\sum^{k-e}_{j=0}q^{(u-r+t)t+(u-k+j)j+(r-t-e)(k-j-e)}}\\
&&\times\qbinom{m-u}{t}\qbinom{m-u}{j}\qbinom{r-t-s}{e-s}
\qbinom{u-r+t}{k-j-e}
C_{(r-t,s)}((\mathfrak{x},\xi), (\mathfrak{y},\eta)).
\end{eqnarray*}
Substituting (\ref{eq:C_rs}) for
$C_{(r-t,s)}((\mathfrak{x},\xi), (\mathfrak{y},\eta))$ and
using Proposition \ref{thm:q-bino} (\ref{eq:most_important}) (a) as
$(k,x,r,t)=(m-s,u-s,r-s,r-t-s)$, i.e.,
\[
\qbinom{u-s}{r-t-s}\qbinom{m-u}{t}
=
\sum^{r}_{i=s}
(-1)^{i-r+t}q^{-t(u-r+t)+\binom{i-r+t}{2}}
\qbinom{i-s}{r-t-s}\qbinom{m-i}{r-i}\qbinom{u-s}{i-s},
\]
and
substituting (\ref{eq:C_rs}) for
$C_{(i,s)}((\mathfrak{x},\xi), (\mathfrak{y},\eta))$ again,
$C_{(r,s)}C_{(k,h)}((\mathfrak{x},\xi), (\mathfrak{y},\eta))$
is equal to
\begin{eqnarray*}
\lefteqn{q^{ml}\sum^{r}_{i=s}\sum^{r}_{e=s}
\qbinom{n-r-k+e}{m-r-k+e}\sum^{r-e}_{t=0}(-1)^{i-r+t}
\qbinom{r-t-s}{e-s}\qbinom{i-s}{r-t-s}\qbinom{m-i}{r-i}}\\
&&\times q^{(r-t-e)(k-e)+\binom{i-r+t}{2}}\sum^{k-e}_{j=0}
q^{j(u-r+t-k+j+e)}\qbinom{m-u}{j}\qbinom{u-r+t}{k-j-e}
C_{(i,s)}((\mathfrak{x},\xi), (\mathfrak{y},\eta)).
\end{eqnarray*}
Using Proposition \ref{thm:q-bino} (\ref{eq:q-Vandermonde}) (b)
for the summation of $j$, i.e.,
\[
\sum^{k-e}_{j=0}
q^{j(u-r+t-k+j+e)}\qbinom{m-u}{j}\qbinom{u-r+t}{k-j-e}
=
\qbinom{m-r+t}{k-e},
\]
using Proposition \ref{thm:q-bino}
(\ref{eq:most_important}) (b) as
$(k,x,r,t)=(m-s,i-s,m-s-k+e,e-s)$,
i.e.,
\begin{eqnarray*}
&&\sum^{r-e}_{t=0}(-1)^{i-r+t}
q^{(r-t-e)(k-e)+\binom{i-r+t}{2}}
\qbinom{r-t-s}{e-s}\qbinom{i-s}{r-t-s}
\qbinom{m-r+t}{k-e}\\
&&=
(-1)^{i-e}q^{\binom{i-e}{2}}
\qbinom{i-s}{e-s}\qbinom{m-i}{m-k},
\end{eqnarray*}
and using Proposition \ref{thm:q-bino}
(\ref{eq:most_important}) (a) as
$(k,x,r,t)=(n-r-k+i,i-s,m-r-k+i,0)$,
i.e.,
\[
\sum^{r}_{e=s}(-1)^{i-e}q^{\binom{i-e}{2}}
\qbinom{n-r-k+e}{m-r-k+e}\qbinom{i-s}{e-s}
=
\qbinom{n-r-k+s}{m-r-k+i}q^{(m-r-k+i)(i-s)},
\]
we have
\begin{eqnarray*}
&&\sum^{r}_{e=s}
\qbinom{n-r-k+e}{m-r-k+e}\sum^{r-e}_{t=0}(-1)^{i-r+t}
\qbinom{r-t-s}{e-s}\qbinom{i-s}{r-t-s}\\
&&\times q^{(r-t-e)(k-e)+\binom{i-r+t}{2}}\sum^{k-e}_{j=0}
q^{j(u-r+t-k+j+e)}\qbinom{m-u}{j}\qbinom{u-r+t}{k-j-e}\\
&&=
\qbinom{m-i}{m-k}
\qbinom{n-r-k+s}{m-r-k+i}q^{(m-r-k+i)(i-s)}.
\end{eqnarray*}
Therefore the desired result follows.
\end{proof}

\section{Relations between association schemes based on attenuated
spaces and $m$-flat association schemes}
\label{sec:m-flat}
Assume that integers $n$ and $m$ satisfy $0 < m < n$.
First in this section,
we describe the definition of $m$-flat association schemes.
The cosets of $\mathbb{F}_q^n$ relative to any $m$-dimensional vector
subspace are called {\it $m$-flats}.
Let $X_n(m)$ be the set of all $m$-flats of $\mathbb{F}_q^n$.
Then the cardinality of $X_n(m)$ is $q^{n-m}\qbinom{n}{m}$.
We define the relation $T_{(i,j)}$ on $X_n(m)$
to be the set of pairs 
$(\mathfrak{p} + x, \mathfrak{q} + y)$ satisfying
$\dim \mathfrak{p} \cap \mathfrak{q}=m-i$,
for $0\le i\le d$,
and $x-y \in \mathfrak{p}+\mathfrak{q}$ if $j=0$,
$x-y \notin \mathfrak{p}+\mathfrak{q}$ if $j=1$.
Then the pair
$\tilde{J}_q(n,m)=
(X_n(m),\{T_{(i,j)}\}_{0\le i\le d, 0\le j \le (n-m-i)\wedge 1})$
is a symmetric association scheme and called an
{\it $m$-flat association scheme}.
Zhu and Li computed all intersection numbers of $m$-flat association
schemes~\cite{Zhu1997coa} and the author gave the character table
of $m$-flat association schemes~\cite{Kurihara2009cto}.

There are relations between $m$-flat association
schemes and association schemes based on attenuated spaces
as follows:
\begin{thm}
\label{thm:m-flat}
The $m$-flat association scheme $\tilde{J}_q(n,m)$ is isomorphic to
the association scheme based on attenuated space
$\mathfrak{X}(\mathcal{M}_q(n-m,0;n+1,n))$.
\end{thm}
We note that
if $l=1$, the attenuated space is just an affine space,
and
Theorem \ref{thm:m-flat} implies that
Theorem \ref{thm:main} contains the main result of the
paper \cite{Kurihara2009cto} as a special case.

\begin{proof}[Proof of Theorem \ref{thm:m-flat}]
Let $\mathscr{L}(x_1,x_2,\ldots ,x_r)$ denote the subspace
of $\mathbb{F}_q^{n}$ spanned by
vectors $x_1,x_2,\ldots ,x_r\in \mathbb{F}_q^{n}$.
Let $e_i$ be the $i$-th standard base of $\mathbb{F}_q^n$ with $1$
in the $i$-th component and $0$ elsewhere,
and $(\cdot ,\cdot )$ be the standard non-degenerate symmetric
bilinear form on $\mathbb{F}_q^{n}$
(i.e., $(\sum^{n}_{i=1}c_i e_i,\sum^{n}_{i=1}c'_i e_i)
=\sum^{n}_{i=1}c_i c'_i$).
For a subspace $\mathfrak{p}$ of $\mathbb{F}_q^n$, we define
$\mathfrak{p}^{\bot }=\set{\beta \in \mathbb{F}_q^n}
{(\alpha ,\beta )=0\ (\forall \alpha \in \mathfrak{x})}$.
We regard $\mathbb{F}_q^n\oplus \mathbb{F}_q$
as $\mathbb{F}_q ^{n+1}$
and also regard $\{0\} \oplus \mathbb{F}_q$ as $\mathfrak{e}$.
For $\mathfrak{p}+x$ in $X_n(m)$,
we set $\mathfrak{p}_{x} =\set{(v, (x,v))}
{v\in \mathfrak{p}^{\bot}}$
in $\mathbb{F}_q ^{n+1}$.
Since any $x$ and $x'$ in $\mathfrak{p}+x$ satisfy
$(x,v)=(x',v)$,
the map 
$\Phi : \mathfrak{p}+x \mapsto \mathfrak{p}_{x}$
from $X_n(m)$ to $\mathbb{F}_q ^{n+1}$
is well-defined.
Then we immediately check that
$\dim \mathfrak{p}_{x}
=\dim \mathfrak{p}^{\bot }
=n-m$
and
$\mathfrak{p}_{x} \cap \mathfrak{e} =\{(0,0)\}$,
i.e., $\mathfrak{p}_{x} \in \mathcal{M}_q(n-m,0;n+1,n)$,
and
$\Phi$ is an injective map
from $X_n(m)$ to $\mathcal{M}_q(n-m,0;n+1,n)$.
On the other hand, we have
$|X_n(m)|=|\mathcal{M}_q(n-m,0;n+1,n)| =q^{n-m}\qbinom{n}{m}$.
This means that $\Phi$ is bijective.	

Finally, we check that
for each $\mathfrak{p}+x,\mathfrak{q}+y \in X_n(m)$,
$(\mathfrak{p}+x,\mathfrak{q}+y) \in T_{(i,j)}$ if and only if
$(\mathfrak{p}_x,\mathfrak{q}_y) \in R_{(i,j)}$.
By
\begin{eqnarray*}
\dim
(\mathfrak{p}_{x} +\mathfrak{e})/\mathfrak{e}
\cap
(\mathfrak{q}_{y} +\mathfrak{e})/\mathfrak{e}
&=&
\dim
\mathfrak{p}^{\bot} \cap \mathfrak{q}^{\bot}\\
&=&n-\dim (\mathfrak{p} + \mathfrak{q}),
\end{eqnarray*}
it follows that
$\dim \mathfrak{p} \cap \mathfrak{q}=m-i$
if and only if $\dim
(\mathfrak{p}_{x} +\mathfrak{e})/\mathfrak{e}
\cap
(\mathfrak{q}_{y} +\mathfrak{e})/\mathfrak{e}
=n-m-i$.
Moreover, we have
\begin{eqnarray*}
\dim \mathfrak{p}_{x} \cap \mathfrak{q}_{y}
&=&
\dim \set{(v, (x,v))}{v\in
\mathfrak{p}^{\bot} \cap \mathfrak{q}^{\bot}
\ \text{and}\ 
(x,v)=(y,v)}\\
&=&
\begin{cases}
\dim\fp^\bot\cap\fq^\bot&\text{if $\fp^\bot\cap\fq^\bot\subset
\mathscr{L}(x-y)^\bot$,}\\
\dim\fp^\bot\cap\fq^\bot-1&\text{otherwise}
\end{cases}
\\ &=&
\begin{cases}
n-m-i&\text{if $x-y\in\fp+\fq$,}\\
n-m-i-1&\text{otherwise.}
\end{cases}
\end{eqnarray*}
Therefore the desired result follows.
\end{proof}

\appendix
\section{Some formulas}
\label{sec:formula}
\begin{prop}[cf.~\cite{Dunkl1978ats}]
\label{prop:dunkl}
Let $\mathfrak{p},\mathfrak{q}$ be subspaces of $\mathbb{F}_q^n$
with $\dim \mathfrak{p} =a$,
$\dim \mathfrak{q} =b$,
$\dim \mathfrak{p} \cap \mathfrak{q}= x$.
For $c,y\in \mathbb{Z}$ with $x\le y\le a$,
$b\le c \le n-a+y$, the number of subspaces $\mathfrak{w}$ of
$\mathbb{F}_q^n$ with $\mathfrak{w} \supset \mathfrak{q}$,
$\dim \mathfrak{w} =c$ and 
$\dim \mathfrak{w} \cap \mathfrak{p}= y$ is
\[q^{(a-y)(c-b-y+x)}\qbinom{a-x}{y-x}\qbinom{n-b-a+x}{c-b-y+x}.\]
\end{prop}

\begin{cor}[cf.~{\cite[Lemma~9.3.2 (iii)]{Brouwer1989dg}}]
\label{lem:BCN}
If $\mathfrak{p}$ is an $a$-dimensional subspace of $\mathbb{F}_q^n$,
then there are precisely
$q^{(a-y)(c-y)}\qbinom{a}{y}\qbinom{n-a}{c-y}$
$c$-dimensional subspaces $\mathfrak{w}$ of $\mathbb{F}_q^n$ with
$\dim \mathfrak{p}\cap \mathfrak{w}=y$.
\end{cor}

\begin{prop}
\label{thm:q-bino}
\begin{enumerate}
\item \label{eq:q-bino_change}
\begin{enumerate}
\item $\displaystyle
\qbinom{k}{r}\qbinom{r}{t}
=
\qbinom{k}{t}\qbinom{k-t}{r-t}$
\item $\displaystyle
\qbinom{k}{r}\qbinom{r}{t}
=
\qbinom{k}{r-t}\qbinom{k-r+t}{t}$
\end{enumerate}
\item \label{eq:q-bino_orthogonal}
$\displaystyle 
\sum^{k}_{i=r}(-1)^{k-i}q^{\binom{k-i}{2}}\qbinom{k}{i}\qbinom{i}{r}=
\delta _{k,r}$
\item \label{eq:q-Vandermonde}
\begin{enumerate}
\item $\displaystyle
\qbinom{x+y}{k}=
\sum^{k}_{i=0}q^{(x-i)(k-i)}\qbinom{x}{i}\qbinom{y}{k-i}$
\item $\displaystyle
\qbinom{x+y}{k}=
\sum^{k}_{i=0}q^{i(y-k+i)}\qbinom{x}{i}\qbinom{y}{k-i}$
\end{enumerate}
\item \label{eq:most_important}
\begin{enumerate}
\item $\displaystyle
\qbinom{x}{t}\qbinom{k-x}{r-t}=
\sum^{r}_{i=t}(-1)^{i-t}q^{-(r-t)(x-t)+\binom{i-t}{2}}\qbinom{i}{t}
\qbinom{k-i}{r-i}\qbinom{x}{i}$
\item $\displaystyle
\qbinom{x}{t}\qbinom{k-x}{r-t}=
\sum^{r}_{i=t}(-1)^{i-t}q^{(i-t)(k-x-r+t)+\binom{i-t+1}{2}}
\qbinom{i}{t}\qbinom{k-i}{r-i}\qbinom{x}{i}$
\end{enumerate}
\end{enumerate}
\end{prop}
\begin{proof}
(1) Immediate from the definition of the Gaussian coefficient.

(2) and (3) are well known
(cf.~\cite{Corcino2008p}
and \cite[Theorem 3.4]{Andrews1998top}).

(4)
First, to prove Proposition \ref{thm:q-bino} (4) (a) and (b),
we use the following claim.
Remark that the equality between
the first term and the second term
of the claim is also proved by Lv and Wang~\cite{Lv2010eoq}.
\begin{clm}
\[\qbinom{n-p}{m}
=
\sum^{m}_{k=0}(-1)^k q^{-mp+\binom{k}{2}}
\qbinom{p}{k}\qbinom{n-k}{m-k}
=
\sum^{m}_{k=0}(-1)^k q^{(n-m-p)k+\binom{k+1}{2}}
\qbinom{p}{k}\qbinom{n-k}{m-k}\]
\end{clm}
\begin{proof}
There is the following equation (cf.~\cite{Lv2010eoq}):
\begin{equation}
\label{eq:minus_binom}
\qbinom{n-p}{m}
=
(-1)^m q^{m(n-p)-\binom{m}{2}}\qbinom{p-n+m-1}{m}.
\end{equation}
By Proposition \ref{thm:q-bino} (\ref{eq:q-Vandermonde}) (a),
we have
\[
\qbinom{p-n+m-1}{m}
=
\sum^{m}_{k=0}q^{(-n+m-1-k)(m-k)}\qbinom{-n+m-1}{k}\qbinom{p}{m-k}.
\]
Again, we apply (\ref{eq:minus_binom}) for $\qbinom{-n+m-1}{k}$ as
$\qbinom{-n+m-1}{k}
=
(-1)^k q^{k(-n+m-1)-\binom{k}{2}}\qbinom{n-m+k}{k}$.
Consequently,
it follows that
\begin{eqnarray*}
\qbinom{n-p}{m}
&=&
(-1)^m q^{m(n-p)-\binom{m}{2}}
\sum^{m}_{k=0}q^{(-n+m-1-k)(m-k)}\qbinom{p}{m-k}\\
&&\times(-1)^k q^{k(-n+m-1)-\binom{k}{2}}\qbinom{n-m+k}{k}\\
&=&
\sum^{m}_{k=0}
(-1)^{m-k}q^{-m p+\binom{m-k}{2}}
\qbinom{n-m+k}{k}\qbinom{p}{m-k}\\
&=&
\sum^{m}_{k=0}
(-1)^k q^{-m p+\binom{k}{2}}
\qbinom{n-k}{m-k}\qbinom{p}{k}.
\end{eqnarray*}
Similarly, using
Proposition \ref{thm:q-bino} (\ref{eq:q-Vandermonde}) (b)
instead of
Proposition \ref{thm:q-bino} (\ref{eq:q-Vandermonde}) (a),
we obtain
$\qbinom{n-p}{m}
=\sum^{m}_{k=0}(-1)^k q^{(n-m-p)k+\binom{k+1}{2}}
\qbinom{p}{k}\qbinom{n-k}{m-k}$.
\end{proof}
Let us return to the proof of
Proposition \ref{thm:q-bino} (\ref{eq:most_important}) (a).
Using the above claim,
we have
\begin{eqnarray*}
\qbinom{k-x}{r-t} &=& \qbinom{(k-t)-(x-t)}{r-t}\\
&=&
\sum^{r-t}_{i'=0}(-1)^{i'}q^{-(r-t)(x-t)+\binom{i'}{2}}
\qbinom{k-t-i'}{r-t-i'}\qbinom{x-t}{i'}\\
&=&
\sum^{r}_{i=t}(-1)^{i-t}q^{-(r-t)(x-t)+\binom{i-t}{2}}
\qbinom{k-i}{r-i}\qbinom{x-t}{i-t}.
\end{eqnarray*}
Therefore it follows that
\begin{eqnarray*}
\qbinom{x}{t}\qbinom{k-x}{r-t}
&=&
\sum^{r}_{i=t}(-1)^{i-t}q^{-(r-t)(x-t)+\binom{i-t}{2}}
\qbinom{k-i}{r-i}\qbinom{x}{t}\qbinom{x-t}{i-t}\\
&=&
\sum^{r}_{i=t}(-1)^{i-t}q^{-(r-t)(x-t)+\binom{i-t}{2}}
\qbinom{k-i}{r-i}\qbinom{x}{i}\qbinom{i}{t}.
\end{eqnarray*}
In the last line,
we used
Proposition \ref{thm:q-bino} (\ref{eq:q-bino_change}) (a).
Therefore the desired result follows.
Similarly, we can prove
Proposition \ref{thm:q-bino} (\ref{eq:most_important}) (b).
\end{proof}
Note that Proposition \ref{thm:q-bino} (\ref{eq:most_important})
implies a $q$-analogue of the identity (4.26)
in Delsarte~\cite{Delsarte1973aat}.

\begin{lem}
\label{lem:eberlein}
\[
\sum^{m}_{k=0}\qbinom{m-k}{t}E_k(n,m;q;x)=
q^{(m-t)x}\qbinom{m-x}{m-t}\qbinom{n-t-x}{m-t}
\]
\end{lem}
\begin{proof}
By the definition of the generalized Eberlein polynomial,
the left hand side is written in
\begin{equation}
\label{eq:Ebertrans}
\sum^{m}_{k=0}\qbinom{m-k}{t}
\sum^{k}_{j=0}(-1)^{k-j}q^{j x+\binom{k-j}{2}}
\qbinom{m-j}{m-k}\qbinom{m-x}{j}\qbinom{n-m+j-x}{j}.
\end{equation}
Interchanging the first summation and the second summation of
(\ref{eq:Ebertrans}),
and using Proposition
\ref{thm:q-bino} (\ref{eq:q-bino_orthogonal}),
it follows that (\ref{eq:Ebertrans}) is equal to
\begin{eqnarray*}
\sum^{m}_{j=0}q^{j x}\qbinom{m-x}{j}\qbinom{n-m+j-x}{j}
\sum^{m}_{k=j}(-1)^{k-j}q^{\binom{k-j}{2}}\qbinom{m-k}{t}
\qbinom{m-j}{m-k}
\\
=q^{(m-t)x}\qbinom{m-x}{m-t}\qbinom{n-t-x}{m-t}.
\end{eqnarray*}
Therefore the desired result follows.
\end{proof}
Note that Lemma \ref{lem:eberlein}
implies a $q$-analogue of the identity (36)
in Tarnanen-Aaltonen-Goethals~\cite{Tarnanen1985njs}.

\begin{lem}
\label{lem:relation_of_krawtchouk}
\begin{equation}
\label{eq:lem_kraw}
\sum^{r\wedge(u-v)}_{j=0\vee (v-r)}q^{j(v-r+j)}
\qbinom{u-v}{j}\qbinom{v}{r-j}K_s(r,l;q;j)
=\qbinom{u-s}{r-s}K_s(u,l;q;u-v)
\end{equation}
\end{lem}
\begin{proof}
By the definition of the generalized Krawtchouk polynomial
and Proposition \ref{thm:q-bino} (\ref{eq:q-bino_change}) (a),
the left hand side of (\ref{eq:lem_kraw}) is written in
\[
\sum^{s}_{t=0}(-1)^{s-t}q^{t l+\binom{s-t}{2}}\qbinom{r-t}{r-s}
\sum^{r\wedge(u-v)}_{j=0\vee (v-r)}q^{j(v-r+j)}\qbinom{u-v}{j}
\qbinom{v}{t}\qbinom{v-t}{r-j-t}.
\]
By Proposition \ref{thm:q-bino} (\ref{eq:q-Vandermonde}) (b),
we obtain
\[
\sum^{r\wedge(u-v)}_{j=0\vee (v-r)}q^{j(v-r+j)}\qbinom{u-v}{j}
\qbinom{v-t}{r-j-t}=\qbinom{u-t}{r-t}.
\]
Using Proposition \ref{thm:q-bino} (\ref{eq:q-bino_change}) (b),
the left hand side of (\ref{eq:lem_kraw})
is equal to 
\[
\qbinom{u-s}{r-s}\sum^{s}_{t=0}(-1)^{s-t}q^{t l+\binom{s-t}{2}}
\qbinom{u-t}{u-s}\qbinom{u-(u-v)}{t}.
\]
Therefore the desired result follows.
\end{proof}

\noindent
\textbf{Acknowledgments.} 
The author is supported by the fellowship of the Japan Society for
the Promotion of Science. 
The author would like to thank Eiichi Bannai,
Akihiro Munemasa
and Kaishun Wang
for useful discussions and comments.

\end{document}